\documentclass[11pt,twoside,a4paper,reqno]{article}
\usepackage[english]{babel}
\usepackage[latin1]{inputenc}
\usepackage[T1]{fontenc} 
\usepackage{mathrsfs,amssymb,amsmath,amsthm,amscd,bbm,bm}
\usepackage{dsfont,enumitem}
\usepackage[left = 1.15in, right = 1.15 in]{geometry}
\usepackage{xcolor}
\usepackage{cite}
\usepackage{float} 
\usepackage[colorlinks,linkcolor=blue,citecolor=blue,pagebackref,hypertexnames=false, breaklinks]{hyperref}
\newtheorem{theorem}{Theorem}[section]
\newtheorem{lemma}[theorem]{Lemma}
\newtheorem{proposition}[theorem]{Proposition}
\newtheorem{corollary}[theorem]{Corollary}

\newcommand{\para}{\mathbin{\!/\mkern-5mu/\!}}

\title{Yet another heat semigroup characterization of BV functions on Riemannian manifolds}
\author{Patricia Alonso Ruiz\footnote{Partly supported by the NSF grant DMS-1951577}\,{ \& }Fabrice Baudoin \footnote{Partially funded by NSF DMS-1901315}}

\begin{document}

\maketitle

\vspace*{-1em}

\abstract{This paper provides a characterization of functions of bounded variation (BV) in a compact Riemannian manifold in terms of the short time behavior of the heat semigroup. In particular, the main result proves that the total variation of a function equals the limit characterizing the space BV. The proof is carried out following two fully independent approaches, a probabilistic and an analytic one. Each method presents different advantages. 

\tableofcontents

\section{Introduction}
Let $\mathbb M$ be a compact, connected, smooth Riemannian manifold of dimension $n$ with Riemannian volume measure $\mu$. Denoting by $\Delta$ the Laplace-Beltrami operator and by $P_t=e^{t \Delta}$ the heat semigroup on $\mathbb M$, the goal of this paper is to prove the following heat semigroup characterization of the class of $BV$ functions.

\begin{theorem}\label{main}
Let $f\in L^1(\mathbb{M},\mu)$. Then, $f\in BV( {\mathbb M})$ if and only if
\begin{equation}\label{main:limsup}
\limsup_{t \to 0^+} \frac{1}{\sqrt{t}}  \int_{\mathbb M} P_t (|f-f(x)|)(x) d\mu(x)<\infty.
\end{equation}
Moreover, when the latter limsup is finite, the limit exists and satisfies
\begin{equation}\label{main:lim}
\lim_{t \to 0^+} \frac{1}{\sqrt{t}}  \int_{\mathbb M} P_t (|f-f(x)|)(x) d\mu(x) =  \frac{ 2 }{\sqrt{\pi}} \| Df \| (\mathbb M),
\end{equation}
where $\| Df \| (\mathbb M)$ denotes the total variation of $f$.
\end{theorem}

The analogue of Theorem~\ref{main} in the Euclidean case $\mathbb{M}=\mathbb{R}^n$ was first proved by M.~Miranda, Jr., D.~Pallara, F.~Paronetto, and M.~Preunkert in~\cite{MPPP07b}.
The characterization~\eqref{main:limsup} of the class of BV functions can be deduced from known results that hold in more general settings in the context of metric measure spaces, see the discussion in Section~\ref{S:BV_intro}. The main contribution of the present paper is thus to prove in the manifold setting that if the $\limsup$ in~\eqref{main:limsup} is finite, then the limit actually exists and is given by \eqref{main:lim}.  

\medskip

Our motivation to study the existence of the limit~\eqref{main:lim} comes from a larger perspective following the recent works~\cite{ARB20, ABCRST1, ABCRST2, ABCRST3}. In those works a basic idea is to take~\eqref{main:limsup} as a definition of a BV function and deduce from that definition the usual embedddings and functional inequalities satisfied by BV functions. One of the most appealing features of a characterization like~\eqref{main:limsup} only requires a heat semigroup, not a a gradient or even a heat kernel, and it therefore makes sense in the general context of Dirichlet spaces, including many fractals and infinite-dimensional spaces.

\medskip

We propose two fully independent approaches to prove Theorem~\ref{main}: The first one is probabilistic and employs tools from stochastic differential geometry and martingale theory, see Section~\ref{proba approach}. The advantage of this approach is that 
the dimension of the underlying space $\mathbb{M}$ plays no crucial role. As a consequence, this method might be applicable to characterize BV functions in the Wiener space, see~\cite{FM01}. However, it has the drawback that so far, using that method, we could only prove the existence of the limit~\eqref{main:lim} when the function $f$ is in the Sobolev space $W^{1,2}(\mathbb{M})$.

\medskip

On the contrary, the second approach to the proof of Theorem~\ref{main} presented in Section~\ref{S:Analytic_approach} is analytic and yields the result for any $f\in BV( {\mathbb M})$. The method has its roots in the paper~\cite{BBM01} by J.~Bourgain, H.~Brezis, and P.~Mironescu, and the main idea consists in thinking of the heat kernel on the manifold as a family of \textit{mollifiers} satisfying good properties; see also the recent works by G. Leoni and D. Spector~\cite{LS11} and by A.~Kreuml and O.~Mordhorst~\cite{KM19} where similar ideas are developed. The drawback in this analytic approach is that the arguments are very specific to Riemannian manifolds and the finite-dimensionality of $\mathbb{M}$ is used in a critical way through heat kernel estimates. 

\medskip

Finally, let us point out that Theorem \ref{main} may be true under weaker assumptions than the compactness of $\mathbb M$. 
There are indeed several parts of the proof that do not strictly require $\mathbb{M}$ to be compact.
For instance, the probabilistic approach would also work as it is when the manifold is complete and has finite volume and bounded Ricci curvature. The analytic approach certainly requires that $\mathbb{M}$ has a bounded geometry (in a sense to be made precise) and a positive injectivity radius and the study of minimal assumptions that ensure the validity of Theorem~\ref{main} might be the subject of future investigations.

\medskip

The paper is organized as follows: Section~\ref{S:Prelims} provides a short summary of the main concepts from Riemannian geometry and heat kernels that will be used in the subsequent sections. To put the main result Theorem~\ref{main} into context, Section~\ref{S:BV_intro} reviews the different characterizations of BV that are available in the manifold setting. Section~\ref{proba approach} develops the probabilistic proof of Theorem~\ref{main}, while Section~\ref{S:Analytic_approach} does the analytic counterpart. Both sections are fully independent and may be read separately. Finally, Section~\ref{S:main_Sobolev} briefly discusses the extension of Theorem~\ref{main} to Sobolev spaces $W^{1,p}(\mathbb{M})$ for $p>1$.

\section{Notations and preliminaries}\label{S:Prelims}

For the reader's convenience and to fix notation, this section records some classical tools from Riemannian geometry that will be used throughout the paper.
\subsection{Calculus on Riemannian manifolds}

In the sequel we will consider $(\mathbb{M},g)$ to  be an $n$-dimensional compact, connected and smooth Riemannian manifold with Riemannian metric $g$.
For each $p\in \mathbb{M}$, we denote by $T_p\mathbb{M}$ the corresponding tangent space and by $\Gamma(T\mathbb{M})$ the space of smooth vector fields. 
For the ease of notation, we write $\langle u,v \rangle=g_p (u,v)$ for any $u,v \in T_p \mathbb{M}$ and $p \in \mathbb M$. Further details concerning the following definitions and statements can be found for instance in~\cite{GHL90}.

\begin{itemize}[leftmargin=1em]
\item \textbf{Volume measure:} The volume element is the unique density $\nu$ (i.e. locally the absolute value of an $n$-form) such that  for any orthonormal basis $u_1, \cdots, u_n$ of $T_p \mathbb M$, $\nu (u_1, \cdots, u_n)=1$. In a system of local coordinates $(x^i)$ one has
\[
d\nu= \sqrt{ \det (g_{ij})} \, | d x^1 \wedge \cdots \wedge dx^n |,
\]
where $g_{ij}= \left\langle \frac{\partial}{\partial x^i}, \frac{\partial}{\partial x^j} \right \rangle$. The volume element induces on $\mathbb{M}$ a Borel measure $\mu$. Since $\mathbb{M}$ is compact, this Borel measure is finite and for normalization reasons, we shall assume throughout the paper and without loss of generality that $\mu (\mathbb{M})=1$.

\item \textbf{Gradient operator:} For $f \in C^1(\mathbb M)$, the gradient of $f$ is the unique vector field $\nabla f \in \Gamma( T\mathbb{M})$ such that for every$X \in \Gamma( T\mathbb{M})$,
\[
df (X)=\langle \nabla f, X \rangle,
\]
where $df$ denotes the differential of $f$.

\item \textbf{Divergence operator:} For every $C^1$ vector field $X \in \Gamma( T\mathbb{M})$, the divergence of $X$ is the unique function ${\rm div} (X)$ such that for every $f \in C^1(\mathbb M)$,
\[
\int_\mathbb{M} f  {\rm div} (X) d\mu =- \int_\mathbb{M} \langle \nabla f  , X \rangle d\mu.
\]

\item \textbf{Laplace-Beltrami operator:} We define the Laplace-Beltrami operator $\Delta $ by $\Delta ={\rm div} \circ \nabla$, i.e. for $f \in C^2 (\mathbb{M})$ 
\[
\Delta f= {\rm div} ( \nabla f).
\]

\item \textbf{Riemannian distance:} For $p,q \in \mathbb{M}$, the distance between $p$ and $q$ is defined as
\[
d(p,q) =\inf_{\gamma} \int_0^1 \| \gamma' (t) \| \, dt,
\]
where the infimum is taken over the set of piecewise $C^1$ functions $\gamma :[0,1] \to \mathbb{M}$ such that $\gamma(0)=p$, $\gamma(1)=q$.

\item \textbf{Geodesics:} A geodesic $\gamma: I \subset \mathbb{R}\to \mathbb M$ is a $C^2$ function such that $\nabla_{\gamma'} \gamma'=0$, where $\nabla$ denotes the Levi-Civita connection of $\mathbb M$. Locally, geodesics are always distance minimizing.

\item \textbf{Exponential map:} Let $p \in \mathbb{M}$. For $u \in T_p \mathbb M$ there exists a unique  geodesic $\gamma: [0,1] \to \mathbb{M}$ satisfying $\gamma(0)=p$, $\gamma'(0)=u$. In that case, we denote $\exp_p (u)=\gamma(1)$. The map $\exp_p: T_p \mathbb{M} \to \mathbb{M}$ is called the exponential map.

\item \textbf{Derivative of the exponential map:} Let $p \in \mathbb{M}$. The derivative of $\exp_p$ at $0 \in T_p \mathbb{M}$ is the identity map. Therefore, $\exp_p$ induces a local diffeomorphism $\exp_p: U_p \subset T_p M \to \exp_p (U_p) \subset \mathbb{M}$, where $U_p$ is an open neighborhood of $p$.

\item \textbf{Injectivity radius:} Since $\mathbb{M}$ is compact, there exists $\varepsilon >0$, such that for every $p \in \mathbb{M}$, $B_p (0 ,\varepsilon) \subset U_p $ where $B_p (0 ,\varepsilon)$ denotes the ball with radius $\varepsilon$ in $T_p \mathbb{M}$ for the norm $g_p$. The supremum of such $\varepsilon$'s is  called the injectivity radius of the manifold $\mathbb M$.

\item \textbf{Exponential coordinates:} The coordinates induced by $\exp_p: U_p \subset T_p \mathbb{M} \to \exp_p (U_p) \subset \mathbb{M}$ are called the exponential (or normal) coordinates.

\item \textbf{Integration in exponential polar coordinates:} Let $p \in \mathbb{M}$ and $\varepsilon >0$ be a constant smaller than the injectivity radius of $\mathbb{M}$. Then, if $f$ is bounded Borel function on metric ball in $\mathbb M$ with center $p$ and radius $\varepsilon >0$, $B(p,\varepsilon)$, then
\[
\int_{B(p,\varepsilon)} f \, d\mu=\int_0^\varepsilon \int_{S^{n-1}} f( \exp_p r u ) \theta_p (r,u) \, dr du,
\]
where $S^{n-1}$ is the unit sphere in $T_p \mathbb{M}$ equipped with its surface area measure $du$. The function $\theta_p$ is the Jacobian of the exponential map $\exp_p$ in exponential polar coordinates and satisfies
\begin{equation}\label{E:local_Jacobian}
\frac{\theta_p (r,u)}{r^{n-1}} \xrightarrow{r\to 0} 1,
\end{equation}
c.f.~\cite[p.166]{GHL90}. Moreover, since $\mathbb{M}$ is compact by assumption, it admits a global lower bound on its Ricci curvature tensor. In particular, one can deduce that there exist constants $C_1,C_2 >0$ such that 
\begin{equation}\label{E:Jacobian_exp_bound}
\theta_p (r,u) \le C_1 e^{C_2 r}
\end{equation}
for every $p, r, u$, see e.g.~\cite[p.171]{GHL90}. 

\item \textbf{Parallel transport:} If $p,q \in \mathbb{M}$ are sufficiently close, there exists a unique length parametrized geodesic $\gamma$ such that $\gamma (0)=p$ and $\gamma( d(p,q)) =q$. The parallel transport $\para_{p,q} : T_p \mathbb M \to T_q \mathbb M$ is then defined by $\para_{p,q} \xi = X(q)$ for any $\xi \in T_p \mathbb M$, where $X$ is the unique vector field along $\gamma$ such that $X(p)= \xi$ and $\nabla_{\gamma'} X =0$.
\end{itemize}

\subsection{Heat kernel on Riemannian manifolds}

For a general presentation of the heat semigroup theory and heat kernels on Riemannian manifolds, we refer for instance to \cite{baudoin2018geometric} or \cite{Gri09}.

\begin{itemize}[leftmargin=1em]
\item \textbf{Energy form:} The quadratic form given for any $f \in C^1(\mathbb{M})$ by
\[
\mathcal{E}(f,f) =\int_{\mathbb M} \| \nabla f \|^2 d\mu
\]
is closable in $L^2(\mathbb{M},\mu)$. The domain of the closed extension of $\mathcal{E}$ is the Sobolev space $W^{1,2} (\mathbb{M})$.

\item \textbf{$L^2$ Laplacian:} The Laplace-Beltrami operator $\Delta$ in $L^2(\mathbb{M},\mu)$ is essentially self-adjoint on the space $C^\infty (\mathbb{M})$ of smooth functions. It therefore admits a unique self-adjoint extension that will still denote by $\Delta$. The domain of this self-adjoint extension is the Sobolev space $W^{2,2} (\mathbb{M})$.

\item \textbf{Heat semigroup:} The operator $\Delta$ is the generator of a strongly continuous Markov contraction semigroup in $L^2(\mathbb{M},\mu)$ that we denote $(P_t)_{t\ge 0}$ and call the heat semigroup on $\mathbb{M}$.

\item \textbf{Heat kernel:} The heat semigroup $(P_t)_{t\ge 0}$ admits a smooth kernel with respect to the volume measure $\mu$, i.e. there exists a smooth function $p_t(x,y)$, $t >0$, $x,y \in \mathbb{M}$, called the heat kernel such that for every $t >0$ and $f \in L^\infty(\mathbb{M},\mu)$,
\[
P_t f (x) = \int_{\mathbb{M}} p_t (x,y) f(y) d\mu(y).
\]
\item \textbf{Heat kernel asymptotics:} (see~\cite[p.154]{Cha84},\cite{Lud19}). There exists $\kappa >0$ such that 
\begin{equation}\label{E:HK_expansion1}
p_t(x,y)= \frac{1}{(4\pi t)^{n/2}} e^{-\frac{d(x,y)^2}{4t}}\big(u_0(x,y) +t R(t,x,y) \big)
\end{equation}
for any $t\in (0,1)$ and $(x,y)\in\{ (x,y) \in \mathbb{M}\times \mathbb{M}, d(x,y) \le \kappa\}$, where $u_0$ is a continuous function with $u_0(x,x)=1$ and $R(t,x,y)$ is uniformly bounded on $(0,1) \times \mathbb{M}\times \mathbb{M}$.

\item \textbf{Heat kernel estimates (1):} (see \cite[Theorem 5.3.4]{Hsu02}). On a compact Riemannian manifold there exist $C_1,C_2>0$ so that
\begin{equation}\label{E:Hsu_HK_estimate1}
\frac{C_1}{t^{n/2}}e^{-\frac{d(x,y)^2}{4t}}\leq p_t(x,y)\leq \frac{C_2}{t^{(2n-1)/2}}e^{-\frac{d(x,y)^2}{4t}}
\end{equation}
for all $t\in (0,1)$ and $x,y\in\mathbb{M}$. Note that this estimate is also valid on the cut-locus.

\item \textbf{Heat kernel estimates (2):} (Li-Yau estimates, see \cite[Theorem 2.3.5]{baudoin2018geometric}). Since the Ricci curvature of compact manifolds is bounded from below, the heat kernel satisfies the Li-Yau type estimates 
\begin{equation}\label{E:Li_Yau_estimate1}
\frac{C_1}{\mu(B(x,\sqrt{t}))}e^{-\frac{d(x,y)^2}{4(1-\varepsilon)t}}\leq p_t(x,y)\leq \frac{C_2}{\mu(B(x,\sqrt{t}))}e^{-\frac{d(x,y)^2}{4(1+\varepsilon)t}}
\end{equation}
for some $C_1,C_2, \varepsilon>0$ and all $t\in (0,1)$, $x,y\in\mathbb{M}$. 
\end{itemize}

\section{Bounded variation functions on Riemannian manifolds}\label{S:BV_intro}
To motivate and situate the main result Theorem~\ref{main} in the context of existent work concerning functions of bounded variation on Riemannian manifolds, this section provides an overview of the most prominent already known characterizations of the space $BV(\mathbb{M})$. For more detailed expositions we refer to~\cite{MPPP07b,MMS16}. 
Analogous results are available for Sobolev spaces $W^{1,p}(\mathbb{M})$, $p>1$, which for the sake of brevity are not discussed here.

\subsection{Integration by parts}
The definition of BV functions on a Riemannian manifold that builds the parallel to the classical Euclidean case via integration by parts was introduced in~\cite{Mir03} as
\begin{equation}\label{E:BV_IbP_char}
BV(\mathbb{M}):=\{f\in L^1(\mathbb{M},\mu)\colon \|Df\|(\mathbb{M})<\infty\},
\end{equation}
where
\begin{equation}\label{E:def_TVf_mfd}
\|Df\|(\mathbb{M}):=\sup\Big\{\int_{\mathbb{M}}f\,{\rm div}\varphi\,d\mu\colon\varphi\in\Gamma(T^*\mathbb{M}),\|\varphi \|\leq 1\Big\}
\end{equation}
denotes the variation of the function $f$. Here, $\Gamma(T^*\mathbb{M})$ corresponds to the space of differential 1-forms. In analogy to the Euclidean setting, the variation of a function can also be expressed by approximation by Lipschitz functions, giving rise to a further characterization of $BV(\mathbb{M})$ in that
\begin{equation}\label{E:BV_Lip_char}
\|Df\|(\mathbb{M})=\inf\Big\{\liminf_{k\to\infty}\int_{\mathbb{M}}\|\nabla f_k\|\,d\mu\colon \{f_k\}_k\text{ loc. Lipschitz and }\|f_k-f\|_{L^1_{\rm loc}(\mathbb{M},\mu)}\to 0\Big\}
\end{equation}
see e.g.~\cite[Definition 3.1]{Mir03}. The equivalence of both spaces is not stated explicitly there, but it follows for instance from the characterization~\eqref{E:BV_KS_char}, see~\cite{MMS16}. We also refer to~\cite{AdM14} for further versions of~\eqref{E:BV_Lip_char} in the context of metric measure spaces.
\subsection{Korevaar-Schoen class and fractional seminorms}
In their seminal work~\cite{KS93}, Korevaar and Schoen provided a metric characterization of BV functions on Riemann domains in terms of asymptotic behavior of a near-diagonal energy functional. This was later extended in~\cite{MMS16} to complete metric measure spaces satisfying 1-Poincar\'e inequality as
\begin{equation}\label{E:BV_KS_char}
\small
BV(\mathbb{M})=\Big\{f\in L^1(\mathbb{M},\mu)\colon\liminf_{r\to 0^+}\frac{1}{r}\int_{\mathbb{M}}\int_{B(x,r)}\frac{|f(x)-f(y)|}{\sqrt{\mu(B(x,r))}\sqrt{\mu(B(y,r))}}\,d\mu(y)d\mu(x)<\infty\Big\}.
\end{equation}
In general, the limit above may not exist however it is possible to prove, see~\cite[Remark 3.2]{MMS16}, that the $\liminf$ is comparable to the total variation $\|Df\|(\mathbb{M})$. 

\medskip

In a somewhat similar flavor, it has recently been proved in~\cite[Theorem 1.1]{KM19} that the total variation can also be expressed by means of fractional derivatives in the sense that
\begin{equation}\label{E:BV_radial_char}
\|Df\|(\mathbb{M})=\liminf_{r\to 0^+}\int_{\mathbb{M}}\int_{\mathbb{M}}\frac{|f(x)-f(y)|}{d(x,y)}\rho_r(d(x,y))\,d\mu(x)d\mu(y),
\end{equation}
where $\{\rho_r\}_{r>0}$ is a family of radial mollifiers. The analogous statement in the Euclidean setting had originally been asked by Bourgain-Brezis-Mironescu~\cite{BBM01} and answered in~\cite{Dav02}. These results are valid also for $p>1$ and the main result in the present offers recovers them in a shorter fashion.
\subsection{Heat semigroup characterization}
Under the assumption of Ricci curvature bounded from below, a De Giorgi type characterization of $BV(\mathbb{M})$ in terms of the heat semigroup was obtained in~\cite{CM07,MPPP07b} by showing that
\begin{equation}\label{E:BV_sg_char}
\|Df\|(\mathbb{M})=\lim_{t\to 0^+}\int_{\mathbb{M}}|\nabla P_tf(x)|\,d\mu(x).
\end{equation}
Further extensions to different types of manifolds have appeared in~\cite{GP15} and to more general metric measure spaces in~\cite[Section 5]{MMS16}. 

\medskip

The characterization provided in Theorem~\ref{main} arises from the following result proved in~\cite[Theorem 4.4]{ABCRST2} in the context of Dirichlet spaces with Gaussian heat kernel estimates: There exist constants $C,c>0$ such that for any $f\in BV(\mathbb{M})$, 
\begin{multline}\label{E:BV2_limsup}
c\limsup_{t\to 0^+}\frac{1}{\sqrt{t}}\int_{\mathbb{M}}P_t(|f-f(x)|)(x)\,d\mu(x)\leq \|Df\|(\mathbb{M})\\
\leq C\liminf_{t\to 0^+}\frac{1}{\sqrt{t}}\int_{\mathbb{M}}P_t(|f-f(x)|)(x)\,d\mu(x).
\end{multline}
It is noteworthy to mention the open question whether in general the existence of the limsup implies the existence of the limit. This fact is true in the Euclidean setting~\cite[Remark 3.5]{MPPP07a} as a result of its validity in the case of indicator functions of sets of finite perimeter studied by Ledoux in~\cite{Led94}, see also~\cite{Pre04}.
In view of~\eqref{E:BV2_limsup}, we  have
\begin{equation}\label{E:BV_BV2_char}
BV(\mathbb{M})=\Big\{f\in L^1(\mathbb{M},\mu)\colon\limsup_{t\to 0^+}\frac{1}{\sqrt{t}}\int_{\mathbb{M}}P_t(|f-f(x)|)(x)\,d\mu(x)<\infty\Big\}.
\end{equation}
Theorem~\ref{main} in the present paper thus settles the question of convergence of the limit above to the variation of $f$. 
Especially in Section~\ref{S:Analytic_approach} we will also work with the associated heat kernel $p_t(x,y)$ characterization
\begin{equation}\label{E:BV2_Gaussian}
\limsup_{t \to 0^+} \frac{1}{\sqrt{t}}\int_{\mathbb{M}}\int_{\mathbb{M}}|f(x)-f(y)| p_t(x,y) \,d\mu(y)\,d\mu(x) <+\infty
\end{equation}
if and only if $f\in BV(\mathbb{M})$. 

\section{Probabilistic approach}\label{proba approach}

This section proves Theorem~\ref{main} for $f \in W^{1,2}(\mathbb M)$ by using tools from stochastic differential geometry. The whole section can be read independently from section~\ref{S:Analytic_approach}, where Theorem \ref{main} is proved for any $f \in BV(\mathbb M)$ in an analytic manner. 
We thus assume throughout the section that the reader is familiar with tools of stochastic differential geometry such as those found e.g. in~\cite{Hsu02}.

\medskip

In the sequel we denote by $(X_t)_{t\geq 0}$ the symmetric diffusion process generated by $\frac{1}{2} \Delta$, i.e. the Brownian motion on $\mathbb{M}$. Also, the symbol $\circ dX $ will denote the Stratonovitch stochastic derivative and $dX$ the It\^o's stochastic derivative. The latter should not be confused with the exterior derivative of a smooth function $f$ at a point $x \in \mathbb{M}$, which is denoted by $df(x)$. The stochastic parallel transport for the Levi-Civita connection $\nabla$ along the paths of $(X_t)_{t \ge 0}$ will be denoted by $\para_{0,t}$. The map $\para_{0,t}: T_{X_0} \mathbb M \to T_{X_t} \mathbb M$ is an isometry and it is known that the anti-development of $(X_t)_{t \ge 0}$ defined by
 \[
 B_t=\int_0^t \para_{0,s}^{-1} \circ dX_s
 \]
 is a Brownian motion on the tangent space $T_{X_0}\mathbb M$. Consider then the process $\tau_t:T_{X_t}^*\mathbb{M}\rightarrow T^*_{X_0}\mathbb{M}$ which is the solution to the covariant Stratonovitch stochastic differential equation
\begin{equation}\label{tau_t}
\circ d[\tau_t \alpha(X_t)]=\tau_t\left( \nabla_{\circ dX_t}-\frac{1}{2}
\mathfrak{Ric} \,  dt\right) \alpha(X_t),~~\tau_0=\mathbf{Id},
\end{equation}
where $\alpha$ is any smooth one-form and $\mathfrak{Ric}$ denotes the Ricci curvature of $\mathbb{M}$. The Stratonovitch integration by parts formula implies
\begin{equation}\label{tau=M Theta}
\tau_t=\mathcal{M}_{t}\Theta_{t},
\end{equation}
where the process $\Theta_t: T_{X_t}^{*}\mathbb{M}\rightarrow T_{X_0}^{*}\mathbb{M}$ is the solution to the covariant Stratonovitch stochastic differential equation
\begin{equation}\label{Theta equation}
\circ d[\Theta_t \alpha(X_t)]= \Theta_t \nabla_{\circ dX_t} \alpha(X_t),~~\Theta_0=\mathbf{Id},
\end{equation}
and the multiplicative functional $(\mathcal{M}_t)_{t\geq 0}$ is the solution to the ordinary differential equation
\begin{equation}\label{multiplicative function M_t}
\frac{d\mathcal{M}_t}{dt}=-\frac{1}{2}\mathcal{M}_t \, \Theta_t \, \mathfrak{Ric} \, \Theta_t^{-1}, ~~\mathcal{M}_0 =\mathbf{Id}.
\end{equation} 

We will also consider the heat semigroup on one-forms given by
\[
Q_t=e^{t \square},
\]
where $\square$ is the Hodge-de Rham Laplacian on $\mathbb{M}$. By the Feynman-Kac formula, for every $L^2$ one-form $\eta$ we have
\begin{align}\label{FKQ}
Q_{t/2}  \eta (x)=\mathbb{E}_x \left( \tau_t \eta (X_t)  \right),
\end{align}
where $\mathbb{E}_x$ denotes the conditional expectation given $X_0=x$. 

\subsection{Smooth  case}

We first note that the proof of Theorem \ref{main} is relatively straightforward if $f \in C^2(\mathbb{M})$.

\begin{theorem}\label{main_smooth}
For any $f \in C^2( {\mathbb M})$,
\[
\lim_{t \to 0^+} \frac{1}{\sqrt{t}}  \int_{\mathbb M} P_t (|f-f(x)|)(x) d\mu(x) =  \frac{ 2 }{\sqrt
{\pi}} \int_{\mathbb{M}} \| d f (x) \| d\mu(x).
\]

\end{theorem}

\begin{proof}
Let $f \in C^2(\mathbb{M})$. From It\^o's formula, one has
\[
f(X_t)=f(X_0)+\frac{1}{2} \int_0^t \Delta f (X_s) ds +\int_0^t \langle d f (X_s) , \para_{0,s}  dB_s \rangle.
\]
Therefore,
\begin{align*}
\int_{\mathbb M} P_{t/2} (|f-f(x)|)(x) d\mu(x)&=  \int_{\mathbb M}  \mathbb{E}_x(|f(X_t)-f(x)|) d\mu(x) \\
 &=\int_{\mathbb M}  \mathbb{E}_x\left(\left|\frac{1}{2} \int_0^t \Delta f (X_s) ds +\int_0^t \langle d f (X_s) , \para_{0,s}  dB_s \rangle\right|\right) d\mu(x).
\end{align*}
Since $f \in C^2(\mathbb{M})$, the dominated convergence theorem implies
\[
\lim_{t \to 0} \frac{1}{t} \int_{\mathbb M}  \mathbb{E}_x\left(\left| \int_0^t \Delta f (X_s) ds \right|\right) d\mu(x)= \int_{\mathbb M} | \Delta f (x) | d\mu(x).
\]
Similarly, one has
\begin{align*}
 & \int_{\mathbb M}  \mathbb{E}_x\left(\left|\int_0^t \langle d f (X_s) , \para_{0,s}  dB_s \rangle -\int_0^t \langle d f (x) ,   dB_s \rangle \right|\right) d\mu(x) \\
 = & \int_{\mathbb M}  \mathbb{E}_x\left(\left|\int_0^t \langle \Theta_s d f (X_s) - d f (x) ,   dB_s \rangle \right|\right) d\mu(x) \\
\le &\left( \int_{\mathbb M}  \mathbb{E}_x\left(\left|\int_0^t \langle \Theta_s d f (X_s) - d f (x) ,   dB_s \rangle \right|^2\right) d\mu(x)\right)^{1/2} \\
= &\left( \int_{\mathbb M}  \mathbb{E}_x\left(\int_0^t \| \Theta_s d f (X_s) - d f (x)\|^2ds\right) d\mu(x)\right)^{1/2}
\end{align*}
and thus, again by the dominated convergence theorem,
\begin{align*}
\lim_{t \to 0^+} \frac{1}{\sqrt{t}} \int_{\mathbb M}  \mathbb{E}_x\left(\left|\int_0^t \langle d f (X_s) , \para_{0,s}  dB_s \rangle -\int_0^t \langle d f (x) ,   dB_s \rangle \right|\right) d\mu(x)=0.
\end{align*}
Therefore, one concludes
\[
\lim_{t \to 0} \frac{1}{\sqrt{t}}  \int_{\mathbb M} P_t (|f-f(x)|)(x) d\mu(x)=\lim_{t \to 0} \frac{1}{\sqrt{t}} \int_{\mathbb M}  \mathbb{E}_x\left(\left|\int_0^{2t} \langle d f (x) ,   dB_s \rangle \right|\right) d\mu(x).
\]
We then note that
\[
\int_{\mathbb M}  \mathbb{E}_x\left(\left|\int_0^{2t} \langle d f (x) ,   dB_s \rangle \right|\right) d\mu(x)=\int_{\mathbb M}  \mathbb{E}_x\left(\left| \langle d f (x) ,   B_{2t} \rangle \right|\right) d\mu(x)
\]
and that, under $\mathbb{P}_x$, the random variable $\langle  d  f (x), B_{2t} \rangle$ is a Gaussian random variable with mean zero and variance $2t \| d  f (x) \|^2$. Therefore
\begin{equation}\label{E:Gaussian_exp}
\mathbb{E}_x \left(| \langle  d  f (x), B_{2t} \rangle | \right)=\sqrt{2t} \| d  f (x) \|\mathbb{E}_x \left(|N |\right),
\end{equation}
where $N$ is a Gaussian random variable with mean zero and variance 1. Since
\[
\mathbb{E}_x \left(|N |\right)=\frac{1}{\sqrt{2\pi}}\int_{-\infty}^{+\infty} |x | e^{-x^2/2} dx = \frac{2}{\sqrt{2\pi}}\int_{0}^{+\infty} x  e^{-x^2/2} dx= \frac{2}{\sqrt{2\pi}},
\]
we conclude
\[
\lim_{t \to 0} \frac{1}{\sqrt{t}}  \int_{\mathbb M} P_t (|f-f(x)|)(x) d\mu(x) =  \frac{ 2 }{\sqrt{\pi}} \int_{\mathbb{M}} \| d f (x) \| d\mu(x).
\]
\end{proof}

\subsection{General $W^{1,2}$ case}
Lowering the regularity of the function $f$ to $W^{1,2}(\mathbb M)$ makes the proof of Theorem~\ref{main} significantly more involved. We subdivide it into several lemmas which may also be useful on their own.
\begin{lemma}\label{lem1}
For any $f \in W^{1,2}(\mathbb M)$,
\[
\lim_{t \to 0} \frac{1}{\sqrt{t}} \| P_t f - f \|_{L^1(\mathbb{M},\mu)}=0.
\]
\end{lemma}

\begin{proof}
Since the volume measure $\mu$ is normalized, 
\begin{equation}\label{E:lem1_01}
\| P_t f - f \|_{L^1(\mathbb{M},\mu)} \le \| P_t f - f \|_{L^2(\mathbb{M},\mu)}.
\end{equation}
The operator $\Delta$ is a non-positive, densely defined, self-adjoint operator on $L^2(\mathbb M, \mu)$ and the spectral theorem thus yields a measure space $(\hat{\mathbb{M}}, \hat{\mu})$, a unitary map $U:  L^2(\mathbb M, \mu) \rightarrow L^2(\hat{\mathbb M}, \hat{\mu})$ and a non-negative real valued measurable function $\lambda$ on $\hat{\mathbb{M}}$ such that
\[
(U \Delta U^{-1}) g (m) =-\lambda (m)  g(m), 
\]
for a.e. $m  \in \hat{\mathbb{M}}$, and $g \in L^2(\hat{\mathbb M}, \hat{\mu})$ such that $U^{-1}g \in {\rm dom}(\Delta )$. We have then 
\[
(U P_t U^{-1}) g (m) =e^{-t\lambda (m)}  g(m)
\]
for a.e. $m\in \hat{\mathbb{M}}$, and $g \in L^2(\hat{\mathbb M}, \hat{\mu})$. In particular, for any $f \in L^2(\mathbb M, \mu)$ we have
\[
\| P_t f - f \|^2_{L^2(\mathbb{M},\mu)}= \int_{\hat{\mathbb M}} (1-e^{-t\lambda (m)})^2 \hat{f} (m)^2 d \hat{\mu}(m),
\]
where $\hat{f}=Uf$ and hence deduce
\[
\frac{1}{t} \| P_t f - f \|^2_{L^2(\mathbb{M},\mu)}= \int_{\hat{\mathbb M}} \frac{(1-e^{-t\lambda (m)})^2}{t\lambda (m) }   \lambda (m) \hat{f} (m)^2 d \hat{\mu}(m).
\]
Moreover, for every $t>0$
\[
\sup_{\lambda >0}  \frac{(1-e^{-t\lambda })^2}{t\lambda  }= \sup_{\lambda >0}  \frac{(1-e^{-\lambda })^2}{\lambda  } <+\infty
\]
and $f \in W^{1,2}(\mathbb M)$ implies that $\int_{\hat{\mathbb M}}   \lambda (m) \hat{f} (m)^2 d \hat{\mu}(m) <+\infty$. Since a.e. $m$ it holds that $ \frac{(1-e^{-t\lambda (m)})^2}{t\lambda (m) } \to 0$ as $t\to 0$, by virtue of the dominated convergence theorem we obtain
\[
\lim_{t \to 0} \frac{1}{t} \| P_t f - f \|^2_{L^2(\mathbb{M},\mu)}=0,
\]
which together with~\eqref{E:lem1_01} yields the conclusion.
\end{proof}

\begin{corollary}
For any $f \in W^{1,2}(\mathbb M)$,
\[
\lim_{t \to 0} \frac{1}{\sqrt{t}} \left | \int_{\mathbb M} P_t (|f-f(x)|)(x) d\mu(x) - \int_{\mathbb M} P_t (|f-P_tf(x)|)(x) d\mu(x) \right| =  0.
\]
\end{corollary}

\begin{proof}
Indeed, this follows from Lemma~\ref{lem1} since
\begin{align*}
 &  \left | \int_{\mathbb M} P_t (|f-f(x)|)(x) d\mu(x) - \int_{\mathbb M} P_t (|f-P_tf(x)|)(x) d\mu(x) \right| \\ \le 
    &  \int_{\mathbb M} P_t ( \left | \, |f-f(x)|)(x)- |f-P_tf(x)| \, \right|)(x) d\mu(x) \\
 \le    & \int_{\mathbb M} P_t ( \left | f(x)-P_tf(x) \right|)(x) d\mu(x) \\
 = & \int_{\mathbb M}  \left | f(x)-P_tf(x) \right| d\mu(x) \\
 = &  \| P_t f - f \|_{L^1(\mathbb{M},\mu)}.
\end{align*}
\end{proof}

In order to prove Theorem \ref{main} in the case $f \in  W^{1,2}(\mathbb M)$, it is therefore enough to prove that for $f \in W^{1,2}(\mathbb M)$
 \[
\lim_{t \to 0} t^{-1/2}  \int_{\mathbb M} P_t (|f-P_tf(x)|)(x) d\mu(x) =  \frac{ 2 }{\sqrt
{\pi}} \| Df \| (\mathbb M).
\]
We are now ready to establish the key probabilistic representation.

\begin{lemma}
For any $f \in W^{1,2}(\mathbb M)$ and $t \ge 0$,
\[
 \int_{\mathbb M} P_{t/2} (|f-P_{t/2}f(x)|)(x) d\mu(x)= \int_{\mathbb M} \mathbb{E}_x \left( \left| \int_0^t \langle \Theta_s d P_{t/2-s/2} f (X_s) ,   dB_s \rangle \right| \right) d\mu(x).
 \]
\end{lemma}

\begin{proof}
Let $f \in W^{1,2}(\mathbb M)$. Observe first that due to the ellipticity of the operator $\Delta$, the semigroup $P_t$ has a smooth heat kernel and therefore the function $(t,x) \mapsto P_t f (x)$ is smooth on $(0,+\infty) \times \mathbb M$. Then, fix $t>0$ and consider the process $M_s=(P_{t/2-s/2} f)(X_s)$. It\^o's formula shows that $(M_s)_{0 \le s \le t}$ is a martingale such that
\[
M_s =M_0+\int_0^s \langle d P_{t/2-u/2} f (X_u), \para_{0,u}  dB_u \rangle.
\]
Therefore
\[
 \int_{\mathbb M} \mathbb{E}_x \left( \left| M_s -M_0 \right| \right) = \int_{\mathbb M} \mathbb{E}_x \left( \left| \int_0^s \langle d P_{t/2-u/2} f (X_u), \para_{0,u}  dB_u \rangle \right| \right) .
\]
Applying this identity for $s=t$ and observing that $M_0=P_{t/2}f(X_0)$ and $M_t=f(X_t)$ yields
\[
\int_{\mathbb M}  \mathbb{E}_x(|f(X_t)-P_{t/2}f(x)|)(x) d\mu(x)= \int_{\mathbb M} \mathbb{E}_x \left( \left| \int_0^t \langle  d P_{t/2-s/2} f (X_s) ,  \para_{0,s}  dB_s \rangle \right| \right) d\mu(x).
 \]
 We conclude by noticing that
\[
 \int_{\mathbb M}  \mathbb{E}_x(|f(X_t)-P_{t/2}f(x)|)(x) d\mu(x)=  \int_{\mathbb M} P_{t/2} (|f-P_{t/2}f(x)|)(x) d\mu(x)
\]
and that
\[
\int_0^t \langle  d P_{t/2-s/2} f (X_s) ,  \para_{0,s}  dB_s \rangle  = \int_0^t \langle  \Theta_s d P_{t/2-s/2} f (X_s) ,    dB_s \rangle.
\]
\end{proof}

\begin{lemma}\label{lem4}
For any $f \in W^{1,2}(\mathbb M)$,
\[
\lim_{t \to 0} \frac{1}{\sqrt{t}} \left| \int_{\mathbb M} \mathbb{E}_x \left( \left| \int_0^t \langle \Theta_s d P_{t/2-s/2} f (X_s) ,   dB_s \rangle \right| - \left| \int_0^t \langle \tau_s d P_{t/2-s/2} f (X_s) ,   dB_s \rangle \right| \right) d\mu(x)  \right| =0.
\]
\end{lemma}

\begin{proof}
First, we have
\begin{align*}
 &  \left| \int_{\mathbb M} \mathbb{E}_x \left( \left| \int_0^t \langle \Theta_s d P_{t/2-s/2} f (X_s) ,   dB_s \rangle \right| - \left| \int_0^t \langle \tau_s d P_{t/2-s/2} f (X_s) ,   dB_s \rangle \right| \right) d\mu(x)  \right| \\
 \le &\int_{\mathbb M} \mathbb{E}_x \left( \left| \int_0^t \langle (\Theta_s -\tau_s)d P_{t/2-s/2} f (X_s) ,   dB_s \rangle  \right| \right) d\mu(x) \\
 \le & \left( \int_{\mathbb M} \mathbb{E}_x \left( \left| \int_0^t \langle (\Theta_s -\tau_s)d P_{t/2-s/2} f (X_s) ,   dB_s \rangle  \right| ^2\right) d\mu(x)\right)^{1/2} \\
 =& \left( \int_{\mathbb M} \mathbb{E}_x \left(  \int_0^t \left\| (\Theta_s -\tau_s)d P_{t/2-s/2} f (X_s) \right\|^2 ds\right) d\mu(x)\right)^{1/2} \\
 \le & \left( \int_{\mathbb M} \mathbb{E}_x \left(  \int_0^t \left\| \Theta_s -\tau_s \right\|^2 \, \left\| d P_{t/2-s/2} f (X_s) \right\|^2 ds\right) d\mu(x)\right)^{1/2}.
\end{align*}
Let us now observe that from~\eqref{tau=M Theta} and because $\Theta_s$ is an isometry, it holds that
\[
\left\| \Theta_s -\tau_s \right\| = \left\| \Theta_s -\mathcal{M}_s \Theta_s  \right\| \le \left\| \mathbf{Id} -\mathcal{M}_s   \right\|.
\]
In view of~\eqref{multiplicative function M_t}, we have then
\begin{align*}
\mathbf{Id} -\mathcal{M}_s= \frac{1}{2}\int_0^s \mathcal{M}_u \, \Theta_u \, \mathfrak{Ric} \, \Theta_u^{-1} du.
\end{align*}
Since $\mathbb M$ is compact, there exists a constant $K \ge 0$ such that, in the sense of bilinear forms, $\mathfrak{Ric} \ge -K$ and also $\| \mathfrak{Ric} \| \le K$. Equation \eqref{multiplicative function M_t} and Gronwall's lemma imply $\| \mathcal{M}_u \| \le e^{Ku/2}$ and thus
\begin{align*}
\| \mathbf{Id} -\mathcal{M}_s\| & \le  \frac{1}{2}\int_0^s \| \mathcal{M}_u \, \Theta_u \, \mathfrak{Ric} \, \Theta_u^{-1} \| du \\
& \le  \frac{K}{2}\int_0^s \| \mathcal{M}_u  \| du \\
& \le  \frac{K}{2}\int_0^s e^{Ku/2}  du= e^{Ks/2}-1.
\end{align*}
Therefore, for $s \le t$,
\[
\left\| \Theta_s -\tau_s \right\| \le e^{Ks/2}-1 \le e^{Kt/2}-1
\]
and we conclude that
\begin{align*}
 &  \left| \int_{\mathbb M} \mathbb{E}_x \left( \left| \int_0^t \langle \Theta_s d P_{t/2-s/2} f (X_s) ,   dB_s \rangle \right| - \left| \int_0^t \langle \tau_s d P_{t/2-s/2} f (X_s) ,   dB_s \rangle \right| \right) d\mu(x)  \right|  \\
  \le & (e^{Kt/2}-1)  \left( \int_{\mathbb M} \mathbb{E}_x \left(  \int_0^t  \left\| d P_{t/2-s/2} f (X_s) \right\|^2 ds\right) d\mu(x)\right)^{1/2} \\
  \le & (e^{Kt/2}-1)  \left( \int_{\mathbb M}  \int_0^t  P_{s/2} (\left\| d P_{t/2-s/2} f  \right\|^2)(x) ds d\mu(x)\right)^{1/2} .
 \end{align*}
Notice now that
\[
\frac{d}{ds} P_{s/2} ((P_{t/2-s/2} f)^2) =P_{s/2} (\left\| d P_{t/2-s/2} f  \right\|^2),
\]
which gives
\[
\int_0^t  P_{s/2} (\left\| d P_{t/2-s/2} f  \right\|^2) ds =P_{t/2}(f^2)- (P_{t/2} f)^2\le P_{t/2}(f^2).
\]
Hence,
\[
\left( \int_{\mathbb M}  \int_0^t  P_{s/2} (\left\| d P_{t/2-s/2} f  \right\|^2)(x) ds d\mu(x)\right)^{1/2} \le \| f \|_{L^2(\mathbb{M}, \mu)}
\]
and finally
\begin{multline*}
\left| \int_{\mathbb M} \mathbb{E}_x \left( \left| \int_0^t \langle \Theta_s d P_{t/2-s/2} f (X_s) ,   dB_s \rangle \right| - \left| \int_0^t \langle \tau_s d P_{t/2-s/2} f (X_s) ,   dB_s \rangle \right| \right) d\mu(x)  \right|  \\
\le  (e^{Kt/2}-1)\| f \|_{L^2(\mathbb{M}, \mu)}.
\end{multline*}
It is then clear that
\begin{align*}
  \lim_{t \to 0} \frac{1}{\sqrt{t}} \left| \int_{\mathbb M} \mathbb{E}_x \left( \left| \int_0^t \langle \Theta_s d P_{t/2-s/2} f (X_s) ,   dB_s \rangle \right| - \left| \int_0^t \langle \tau_s d P_{t/2-s/2} f (X_s) ,   dB_s \rangle \right| \right) d\mu(x)  \right| =0.
 \end{align*}
\end{proof}

\begin{lemma}
For any $f \in W^{1,2}(\mathbb M)$,
\begin{align*}
 \lim_{t \to 0} \frac{1}{\sqrt{t}} \left| \int_{\mathbb M} \mathbb{E}_x \left( \left| \int_0^t \langle \tau_s d P_{t/2-s/2} f (X_s) ,   dB_s \rangle \right| - \left| \int_0^t \langle  d P_{t/2} f (X_0) ,   dB_s \rangle \right| \right) d\mu(x)  \right| =0.
 \end{align*}
\end{lemma}

\begin{proof}
Firstly, the integral above can be bounded as
\begin{align*}
 & \left| \int_{\mathbb M} \mathbb{E}_x \left( \left| \int_0^t \langle \tau_s d P_{t/2-s/2} f (X_s) ,   dB_s \rangle \right| - \left| \int_0^t \langle  d P_{t/2} f (X_0) ,   dB_s \rangle \right| \right) d\mu(x)  \right| \\
 \le &\int_{\mathbb M} \mathbb{E}_x \left( \left| \int_0^t \langle \tau_s d P_{t/2-s/2} f (X_s) ,   dB_s \rangle  -  \int_0^t \langle  d P_{t/2} f (X_0) ,   dB_s \rangle \right| \right) d\mu(x) \\
  \le &\left(\int_{\mathbb M} \mathbb{E}_x \left( \left| \int_0^t \langle \tau_s d P_{t/2-s/2} f (X_s) ,   dB_s \rangle  -  \int_0^t \langle  d P_{t/2} f (X_0) ,   dB_s \rangle \right|^2 \right) d\mu(x)\right)^{1/2} \\
  =& \left(\int_{\mathbb M} \mathbb{E}_x \left(  \int_0^t \left\| \tau_s d P_{t/2-s/2} f (X_s)   -  d P_{t/2} f (X_0)  \right\|^2 ds \right) d\mu(x)\right)^{1/2} \\
   =& \left(\int_{\mathbb M} \mathbb{E}_x \left(  \int_0^t \left\| \tau_s d P_{t/2-s/2} f (X_s) \right\|^2- 2 \left\langle \tau_s d P_{t/2-s/2} f (X_s) ,d P_{t/2} f (X_0) \right\rangle   + \left\|  d P_{t/2} f (X_0)  \right\|^2 ds \right) d\mu(x)\right)^{1/2}.
\end{align*}
Secondly, we analyze all terms in the latter integral expression separately. For the first one, using the notations and computations in the proof of Lemma~\ref{lem4} we obtain
\begin{align*}
\int_{\mathbb M} \mathbb{E}_x \left(  \int_0^t \left\| \tau_s d P_{t/2-s/2} f (X_s) \right\|^2 ds \right) d\mu(x)
 & \le e^{Kt}\int_{\mathbb M} \mathbb{E}_x \left(  \int_0^t \left\|  d P_{t/2-s/2} f (X_s) \right\|^2 ds \right) d\mu(x) \\
 & \le e^{Kt}\int_{\mathbb M}( P_{t/2}(f^2)(x)- (P_{t/2} f)^2(x)) d\mu(x).
 \end{align*}
 Then, using \eqref{FKQ}, we have
 \begin{align*}
 & \int_{\mathbb M} \mathbb{E}_x \left( \int_0^t \left\langle \tau_s d P_{t/2-s/2} f (X_s) ,d P_{t/2} f (X_0) \right\rangle ds   \right) d\mu(x) \\
=& \int_{\mathbb M}   \int_0^t \left\langle  \mathbb{E}_x (\tau_s d P_{t/2-s/2} f (X_s)) ,d P_{t/2} f (x) \right\rangle ds \,   d\mu(x)  \\
=& \int_{\mathbb M}   \int_0^t \left\langle  Q_{s/2} d P_{t/2-s/2} f (x) ,d P_{t/2} f (x) \right\rangle ds \,   d\mu(x) \\
=& \int_0^t \int_{\mathbb M}    \delta  Q_{s/2} d P_{t/2-s/2} f (x) P_{t/2} f (x)  d\mu(x) \, ds,
\end{align*}
where $\delta$ is the adjoint in $L^2$ of the exterior derivative $d$. Since $\delta  Q_{s/2}= P_{s/2} \delta$ and $\Delta = -\delta d$, we deduce then
\begin{align*}
& \int_{\mathbb M} \mathbb{E}_x \left( \int_0^t \left\langle \tau_s d P_{t/2-s/2} f (X_s) ,d P_{t/2} f (X_0) \right\rangle ds   \right) d\mu(x) \\
&=- \int_0^t \int_{\mathbb M}   P_{s/2} \Delta    P_{t/2-s/2} f (x) P_{t/2} f (x)  d\mu(x) \, ds \\
&=- t \int_{\mathbb M}    \Delta    P_{t/2} f (x) P_{t/2} f (x)  d\mu(x).
 \end{align*}
Finally, for the last term we get
\begin{align*}
  \int_{\mathbb M} \mathbb{E}_x \left(  \int_0^t \left\|  d P_{t/2} f (X_0)  \right\|^2 ds\right) d\mu(x) 
& =t  \int_{\mathbb M} \left\|  d P_{t/2} f (x)  \right\|^2 d\mu(x)\\
& =- t \int_{\mathbb M}    \Delta    P_{t/2} f (x) P_{t/2} f (x)  d\mu(x)
\end{align*}
and conclude
\begin{multline*}
\left| \int_{\mathbb M} \mathbb{E}_x \left( \left| \int_0^t \langle \tau_s d P_{t/2-s/2} f (X_s) ,   dB_s \rangle \right| - \left| \int_0^t \langle  d P_{t/2} f (X_0) ,   dB_s \rangle \right| \right) d\mu(x)  \right|^2 \\
\le   e^{Kt}\int_{\mathbb M}( P_{t/2}(f^2)(x)- (P_{t/2} f)^2(x)) d\mu(x)+t \int_{\mathbb M}    \Delta    P_{t/2} f (x) P_{t/2} f (x)  d\mu(x).
\end{multline*}
To estimate the last term, we apply the spectral theorem as in the proof of Lemma~\ref{lem1}. With the notations of Lemma \ref{lem1},
\begin{multline*}
e^{Kt}\int_{\mathbb M}( P_{t/2}(f^2)(x)- (P_{t/2} f)^2(x)) d\mu(x)+t \int_{\mathbb M}    \Delta    P_{t/2} f (x) P_{t/2} f (x)  d\mu(x) \\
 \le \int_{\hat{\mathbb M}}(e^{Kt}(1- e^{-t\lambda(m)}) -t \lambda(m) e^{-t\lambda(m)}) \hat{f}(m)^2  d\hat{\mu}(m)
\end{multline*}
and therefore
\begin{multline*}
\frac{1}{t} \left| \int_{\mathbb M} \mathbb{E}_x \left( \left| \int_0^t \langle \tau_s d P_{t/2-s/2} f (X_s) ,   dB_s \rangle \right| - \left| \int_0^t \langle  d P_{t/2} f (X_0) ,   dB_s \rangle \right| \right) d\mu(x)  \right|^2  \\
 \le \int_{\hat{\mathbb M}}\frac{e^{Kt}(1- e^{-t\lambda(m)}) -t \lambda(m) e^{-t\lambda(m)}}{t \lambda(m)} \lambda(m) \hat{f}(m)^2  d\hat{\mu}(m).
\end{multline*}
 Since $f \in W^{1,2}(\mathbb{M})$, we have
\[
\int_{\hat{\mathbb M}} \lambda(m) \hat{f}(m)^2  d\hat{\mu}(m) <+ \infty
 \]
and moreover, for every $t \in (0,1]$,
\[
 \sup_{m \in \hat{\mathbb M}}\frac{e^{Kt}(1- e^{-t\lambda(m)}) -t \lambda(m) e^{-t\lambda(m)}}{t \lambda(m)} \le \sup_{\lambda >0}\frac{e^{K}(1- e^{-\lambda}) - \lambda e^{-\lambda }}{ \lambda} <+\infty.
 \]
Thus, by virtue of the dominated convergence theorem, we conclude that
 \[
\lim_{t \to 0} \frac{1}{t} \left| \int_{\mathbb M} \mathbb{E}_x \left( \left| \int_0^t \langle \tau_s d P_{t/2-s/2} f (X_s) ,   dB_s \rangle \right| - \left| \int_0^t \langle  d P_{t/2} f (X_0) ,  dB_s \rangle \right| \right) d\mu(x)  \right|^2=0.
 \]
 \end{proof}
 
We are finally ready to prove the following version of the main theorem.
 
\begin{theorem}\label{main_sectionproba}
For any $f \in W^{1,2}( {\mathbb M})$,
\[
\lim_{t \to 0} \frac{1}{\sqrt{t}}  \int_{\mathbb M} P_t (|f-f(x)|)(x) d\mu(x) =  \frac{ 2 }{\sqrt
{\pi}} \| Df \| (\mathbb M).
\]
\end{theorem}

\begin{proof}
Combining the previous lemmas, we get that for $f \in W^{1,2}(\mathbb M)$
\[
 \lim_{t \to 0} \frac{1}{\sqrt{t}} \left| \int_{\mathbb M} P_{t} (|f-f(x)|)(x) d\mu(x) - \int_{\mathbb M} \mathbb{E}_x \left( \left| \int_0^{2t} \langle  d P_{t} f (X_0) ,   dB_s \rangle \right| \right)d\mu(x) \right|=0.
\] 
However, we have
\[
\int_{\mathbb M} \mathbb{E}_x \left( \left| \int_0^{2t} \langle  d P_{t} f (X_0) ,   dB_s \rangle \right| \right) d\mu(x)= \int_{\mathbb M} \mathbb{E}_x \left(| \langle  d P_{t} f (x), B_{2t} \rangle | \right) d\mu(x).
\]
We now note that, under $\mathbb{P}_x$, the random variable $\langle  d P_{t} f (x), B_{2t} \rangle$ is a Gaussian random variable with mean zero and variance $2t \| d P_t f (x) \|^2$, hence
\[
\mathbb{E}_x \left(| \langle  d P_{t} f (x), B_{2t} \rangle | \right)=\sqrt{2t} \| d P_t f (x) \|\mathbb{E}_x \left(|N |\right),
\]
where $N$ is a Gaussian random variable with mean zero and variance 1. Therefore, we deduce that
\[
\int_{\mathbb M} \mathbb{E}_x \left( \left| \int_0^{2t} \langle  d P_{t} f (X_0) ,  \circ dB_s \rangle \right| \right) d\mu(x)=2 \sqrt{\frac{t}{\pi}} \| d P_t f (x) \| 
\]
and since
\[
\lim_{t \to 0} \int_{\mathbb{M}} \| d P_t f (x) \| d\mu(x)=\| Df \| (\mathbb M)
\]
we finally conclude
\[
\lim_{t \to 0} \frac{1}{\sqrt{t}}  \int_{\mathbb M} P_t (|f-f(x)|)(x) d\mu(x) =  \frac{ 2 }{\sqrt{\pi}} \| Df \| (\mathbb M).
\]
\end{proof}

\section{Analytic approach}\label{S:Analytic_approach}

The probabilistic approach in the previous section only allowed us to obtain Theorem~\ref{main} for functions $f\in W^{1,2}(\mathbb{M})$. In this section we prove the result in its full generality, i.e. for any $f\in BV$, by means of purely analytic methods that are completely independent of those in Section~\ref{proba approach}. 
Another advantage of this approach is that it can be extended to any $p>1$ in a straightforward manner, as will be illustrated in Section~\ref{S:main_Sobolev}. On the down side, the specific setting of Riemannian manifolds plays a crucial role; treating more general underlying spaces will need different tools. 

\medskip

The key idea to prove~\eqref{main:lim} relies on the tight relationship between the heat kernel and a Gaussian kernel on $\mathbb{M}$ in short times. 
\begin{proposition}\label{P:BV2_fullGausian}
For any $f \in BV(\mathbb{M})$,
\begin{equation}\label{E:BV2_fullGaussian}
\limsup_{t \to 0^+} \frac{1}{\sqrt{t}}\int_{\mathbb{M}}\int_{\mathbb{M}}|f(x)-f(y)|  \frac{1}{(4\pi t)^{n/2}} e^{-\frac{d(x,y)^2}{4t}} \,d\mu(y)\,d\mu(x) <+\infty.
\end{equation}
\end{proposition}
\begin{proof}
By virtue of the global bound for the heat kernel~\eqref{E:Hsu_HK_estimate1}, there is a constant $C_1>0$ such that
\begin{equation}\label{E:Hsu_HK_estimate}
\frac{C_1}{t^{n/2}}e^{-\frac{d(x,y)^2}{4t}}\leq p_t(x,y)
\end{equation}
for all $t\in (0,1)$ and $x,y\in\mathbb{M}$, hence~\eqref{E:BV2_fullGaussian} follows from~\eqref{E:BV2_Gaussian}.
\end{proof}

\subsection{Approximating the heat kernel by a Gaussian kernel}
The Gaussian kernel in Proposition~\ref{P:BV2_fullGausian} also appears in asymptotic expansion~\eqref{E:HK_expansion1} given by
\begin{equation}\label{E:HK_expansion}
p_t(x,y)= \frac{1}{(4\pi t)^{n/2}} e^{-\frac{d(x,y)^2}{4t}}\big(u_0(x,y) +t R(t,x,y) \big)
\end{equation}
for any $t\in (0,1)$ and $(x,y)\in\{ (x,y) \in \mathbb{M}\times \mathbb{M}, d(x,y) \le \kappa\}$, where $\kappa>0$, $u_0$ is continuous with $u_0(x,x)=1$ and $R(t,x,y)$ is uniformly bounded on $(0,1) \times \mathbb{M}\times \mathbb{M}$. This expansion will allow us to estimate 
\[
\lim_{t \to 0^+} \frac{1}{\sqrt{t}}  \int_{\mathbb{M}} \int_{\mathbb{M}}|f(x)-f(y)|p_t(x,y)\,d\mu(y)\,d\mu(x)
\]
by replacing the heat kernel $p_t(x,y)$ by the simpler Gaussian kernel 
\begin{equation*}
\tilde{p}_t(x,y):=\frac{1}{(4\pi t)^{n/2}}e^{-\frac{d(x,y)^2}{4t}}, \qquad (t,x,y)\in (0,1) \times \mathbb{M}\times \mathbb{M}.
\end{equation*}
The key observation formulated in the next theorem makes possible to make that replacement. Besides of allowing us to work afterwards with $\tilde{p}_t(x,y)$ instead of $p_t(x,y)$ and take advantage of its explicit expression, it also reduces the analysis to integrals in small balls. Recall that the expansion~\eqref{E:HK_expansion} holds for $x,y\in\mathbb{M}$ with $d(x,y)\leq \kappa>0$.
\begin{theorem}\label{T:tildeHK_to_HK_ve}
For any $0<\varepsilon<\kappa $ and $f \in  BV(\mathbb{M})$,
\begin{equation*}
\lim_{t \to 0^+} \frac{1}{\sqrt{t}}\int_{\mathbb{M}}\int_{\mathbb{M}}|f(x)-f(y)|\,| p_t(x,y) - 1_{B(x,\varepsilon)}(y) \tilde{p}_t(x,y) \, |\,d\mu(y)\,d\mu(x)=0.
\end{equation*}
\end{theorem}

The proof is subdivided in several lemmas. The first states that outside of small balls there is no contribution to the limit.

\begin{lemma}\label{L:Step1}
For any $\varepsilon >0$ and $f \in  BV(\mathbb{M})$
\begin{equation}\label{E:Step1a_tildeHK_to_HK_ve}
\lim_{t \to 0^+} \frac{1}{\sqrt{t}}\int_{\mathbb{M}}\int_{\mathbb{M}\setminus B(x,\varepsilon)}|f(x)-f(y)| p_t(x,y) \,d\mu(y)\,d\mu(x)=0
\end{equation}
and
\begin{equation}\label{E:Step1b_tildeHK_to_HK_ve}
\lim_{t \to 0^+} \frac{1}{\sqrt{t}}\int_{\mathbb{M}}\int_{\mathbb{M}\setminus B(x,\varepsilon)}|f(x)-f(y)| \tilde{p}_t(x,y) \,d\mu(y)\,d\mu(x)=0.
\end{equation}
\end{lemma}

\begin{proof}
Applying the upper bound in~\eqref{E:Hsu_HK_estimate1}, one can bound~\eqref{E:Step1a_tildeHK_to_HK_ve} by
\begin{align*}\label{E:tildeHK_to_HK_ve_01}
& \frac{2}{\sqrt{t}}\int_{\mathbb{M}}|f(x)|\int_{\mathbb{M}\setminus B_{\mathbb{M}}(x,\varepsilon)}\frac{C_2}{t^n}e^{-\frac{d(x,y)^2}{4t}}\,d\mu(y)\,d\mu(x)\notag\\
&\leq \frac{2C_2}{t^{(n+1)/2}}e^{-\frac{\varepsilon}{8t}}\int_{\mathbb{M}}|f(x)|\int_{\mathbb{M}\setminus B_{\mathbb{M}}(x,\varepsilon)}\frac{1}{t^{n/2}}e^{-\frac{d(x,y)^2}{8t}}\,d\mu(y)\,d\mu(x)\notag\\
&\leq \frac{2^{n/2}C_2}{C_1 t^{(n+1)/2}}e^{-\frac{\varepsilon}{8t}}\int_{\mathbb{M}}|f(x)|\int_{\mathbb{M}\setminus B_{\mathbb{M}}(x,\varepsilon)}p_{2t}(x,y)\,d\mu(y)\,d\mu(x)
\leq \frac{2^{n/2}C_2}{C_1 t^{(n+1)/2}}e^{-\frac{\varepsilon}{8t}}\|f\|_{L^1(\mathbb{M},\mu)}
\end{align*}
which vanishes as $t\to 0^+$. For~\eqref{E:Step1b_tildeHK_to_HK_ve} we use the lower bound in~\eqref{E:Hsu_HK_estimate1} to get
\begin{align*}
&\frac{1}{\sqrt{t}}\int_{\mathbb{M}}\int_{\mathbb{M}\setminus B(x,\varepsilon)}|f(x)-f(y)| \tilde{p}_t(x,y) \,d\mu(y)\,d\mu(x)\\
&\leq \frac{2}{\sqrt{t}}\int_{\mathbb{M}}|f(x)|\int_{\mathbb{M}\setminus B(x,\varepsilon)}\frac{1}{(2\pi t)^{n/2}}e^{-\frac{d(x,y)^2}{4t}}d\mu(y)\,d\mu(x)\\
&\leq \frac{2}{\sqrt{t}}e^{-\frac{\varepsilon}{8t}}\int_{\mathbb{M}}|f(x)|\int_{\mathbb{M}\setminus B(x,\varepsilon)}\frac{1}{C_1\pi^{n/2}}p_{2t}(x,y)\,d\mu(y)\,d\mu(x)
\leq \frac{2}{C_1\pi^{n/2} \sqrt{t}}e^{-\frac{\varepsilon}{8t}}\|f\|_{L^1(\mathbb{M},\mu)}
\end{align*}
which again vanishes as $t\to 0^+$. 
\end{proof}
\noindent
The next lemma shows that the second term in the expansion~\eqref{E:HK_expansion} is also negligible as $t\to 0$.
\begin{lemma}\label{L:Step2}
For any $f \in  BV(\mathbb{M})$ and $0<\varepsilon<\kappa$,
\[
\lim_{t \to 0^+} \frac{1}{\sqrt{t}}\int_{\mathbb{M}}\int_{ B(x,\varepsilon)}|f(x)-f(y)|   \frac{1}{(4\pi t)^{n/2}} e^{-\frac{d(x,y)^2}{4t}} t R(t,x,y)  \,d\mu(y)\,d\mu(x)=0.
\]
\end{lemma}

\begin{proof}
By virtue of~\eqref{E:HK_expansion}, the function $R(t,x,y)$ is uniformly bounded on $(0,1)\times\mathbb{M}\times\mathbb{M}$. Thus, there exists $M>0$ such that
\begin{align*}
&\frac{1}{\sqrt{t}}\int_{\mathbb{M}}\int_{ B(x,\varepsilon)}|f(x)-f(y)|   \frac{1}{(4\pi t)^{n/2}} e^{-\frac{d(x,y)^2}{4t}} t R(t,x,y)  \,d\mu(y)\,d\mu(x)\\
&\leq 2M\int_{\mathbb{M}}|f(x)|\int_{B(0,\varepsilon)}\frac{\sqrt{t}}{(4\pi t)^{n/2}}e^{-\frac{d(x,y)^2}{4t}} t  \,d\mu(y)\,d\mu(x).
\end{align*}
Using polar coordinates and the change of variables $\tilde{r}=r/\sqrt{t}$ the latter expression can be written as
\begin{align*}
&2M\int_{\mathbb{M}}|f(x)|\int_0^{\varepsilon/\sqrt{t}}\int_{S^{n-1}}\frac{\sqrt{t}}{(4\pi t)^{n/2}}e^{-\frac{\tilde{r}^2}{4}}\theta_x(\sqrt{t}\tilde{r},u)\,du\,d\tilde{r}\,d\mu(x),
\end{align*}
where $\theta_x$ is the determinant of Jacobian of the exponential map in polar coordinates. Moreover, recall from~\eqref{E:local_Jacobian} that
\begin{equation}\label{E:Jacobian_small_t}
\frac{\theta_x(r\sqrt{t},u)}{(r \sqrt{t})^{n-1}}  \xrightarrow{t\to 0^+} 1
\end{equation}
holds uniformly on $r$ and $u$. Hence, for $t$ small enough
\begin{align*}
&\int_{\mathbb{M}}|f(x)|\int_0^{\varepsilon/\sqrt{t}}\int_{S^{n-1}}\frac{\sqrt{t}}{(4\pi t)^{n/2}}e^{-\frac{\tilde{r}^2}{4}}\theta_x(\sqrt{t}\tilde{r},u)\,du\,d\tilde{r}\,d\mu(x)\\
&\leq \frac{\sqrt{t}}{(4\pi)^{n/2}}\int_{\mathbb{M}}|f(x)|\int_0^{\infty}\int_{S^{n-1}}\tilde{r}^{n-1}e^{-\frac{\tilde{r}^2}{4}}\frac{\theta_x(\sqrt{t}\tilde{r},u)}{(\tilde{r}\sqrt{t})^{n-1}}\,du\,d\tilde{r}\,d\mu(x)\\
&\leq \frac{C_\theta\sqrt{t}}{(4\pi)^{n/2}}\int_{\mathbb{M}}|f(x)|\int_0^{\infty}\int_{S^{n-1}}\tilde{r}^{n-1}e^{-\frac{\tilde{r}^2}{4}}du\,d\tilde{r}\,d\mu(x)\\
&\leq \frac{C_\theta\sqrt{t}}{(4\pi)^{n/2}}\|f(x)\|_{L^1(\mathbb{M},\mu)}\mathcal{H}^{n-1}(S^{n-1})\int_0^\infty \tilde{r}^{n-1}e^{-\frac{\tilde{r}^2}{4}}d\tilde{r}
\end{align*}
which vanishes as $t\to 0^+$.
\end{proof}

We are finally in the position to prove Theorem~\ref{T:tildeHK_to_HK_ve}. 
\begin{proof}[Proof of Theorem~\ref{T:tildeHK_to_HK_ve}]
Let $\delta >0$.  In view of the expansion~\eqref{E:HK_expansion}, there is $\varepsilon_\delta>0$ such that $u(x,y) \le 1+\delta$ for every $x,y \in \mathbb{M}$ with $d(x,y) \le \varepsilon_\delta$. 
In the worst case that $\varepsilon_\delta<\varepsilon$, we can split the integral over $B(x,\varepsilon)$ into 
\begin{multline*}
\frac{1}{\sqrt{t}}\int_{\mathbb{M}}\int_{B(x,\varepsilon)\setminus B(x,\varepsilon_\delta)}|f(x)-f(y)|\,| p_t(x,y) -  \tilde{p}_t(x,y) \, |\,d\mu(y)\,d\mu(x)\\
+\frac{1}{\sqrt{t}}\int_{\mathbb{M}}\int_{B(x,\varepsilon_\delta)}|f(x)-f(y)|\,| p_t(x,y) -  \tilde{p}_t(x,y) \, |\,d\mu(y)\,d\mu(x)=:I_1(t)+I_2(t).
\end{multline*}
Applying the triangle inequality, Lemma~\ref{L:Step1} implies $\limsup_{t\to 0^+}I_1(t)=0$. Thus,
\begin{align*}
&\limsup_{t \to 0^+} \frac{1}{\sqrt{t}}\int_{\mathbb{M}}\int_{B(x,\varepsilon)}|f(x)-f(y)|\,| p_t(x,y) -  \tilde{p}_t(x,y) \, |\,d\mu(y)\,d\mu(x) \\
&=\limsup_{t \to 0^+} \frac{1}{\sqrt{t}}\int_{\mathbb{M}}\int_{B(x,\varepsilon_\delta)}|f(x)-f(y)|\,| p_t(x,y) -  \tilde{p}_t(x,y) \, |\,d\mu(y)\,d\mu(x) \\
&\le  \delta \limsup_{t \to 0^+} \frac{1}{\sqrt{t}}\int_{\mathbb{M}}\int_{B(x,\varepsilon_\delta)}|f(x)-f(y)|  \frac{1}{(4\pi t)^{n/2}} e^{-\frac{d(x,y)^2}{4t}} \,d\mu(y)\,d\mu(x) \\
&\le   \delta \limsup_{t \to 0^+} \frac{1}{\sqrt{t}}\int_{\mathbb{M}}\int_{\mathbb{M}}|f(x)-f(y)|  \frac{1}{(4\pi t)^{n/2}} e^{-\frac{d(x,y)^2}{4t}} \,d\mu(y)\,d\mu(x).
\end{align*}
By virtue of~\eqref{E:BV2_fullGaussian} and since $\delta>0$ is arbitrary we conclude
\[
\lim_{t \to 0} \frac{1}{\sqrt{t}}\int_{\mathbb{M}}\int_{B(x,\varepsilon)}|f(x)-f(y)|\,| p_t(x,y) -  \tilde{p}_t(x,y) \, |\,d\mu(y)\,d\mu(x) =0
\]
which together with~\eqref{E:Step1a_tildeHK_to_HK_ve} yields the desired result.
\end{proof}

According to Theorem~\ref{T:tildeHK_to_HK_ve}, in order to  prove Theorem~\ref{main} it is therefore enough to prove that if $f \in  BV(\mathbb{M})$, then there exists $\varepsilon >0$ such that
\begin{equation}\label{E:tildeHK_to_HK_01_ve2}
\lim_{t \to 0^+} \frac{1}{\sqrt{t}}\int_{\mathbb{M}}\int_{B(x,\varepsilon)}|f(x)-f(y)| \, \tilde{p}_t (x,y) \,d\mu(y)\,d\mu(x)=\frac{ 2 }{\sqrt
{\pi}} \| Df \| (\mathbb M).
\end{equation}

\subsection{Smooth case}\label{S:smooth case ana}
As in the probabilistic approach, we start by pointing out that~\eqref{E:tildeHK_to_HK_01_ve2} is straightforward to obtain when $f \in C^2(\mathbb{M})$.

\begin{proposition}\label{P:smooth case ana}
For every $\varepsilon >0$ and $f \in C^2(\mathbb{M})$,
\[
\lim_{t \to 0^+} \frac{1}{\sqrt{t}}\int_{\mathbb{M}}\int_{B(x,\varepsilon)}|f(x)-f(y)| \, \tilde{p}_t (x,y) \,d\mu(y)\,d\mu(x)=\frac{ 2 }{\sqrt
{\pi}}\int_\mathbb{M} \| \nabla f \| d\mu
\]
\end{proposition}

\begin{proof}
Let $\varepsilon >0$ be smaller than the injectivity radius of $\mathbb{M}$. We proceed as in the proof of Lemma~\ref{L:Step2} and use polar coordinates to write
\begin{align}\label{E:smooth case ana_01}
& \int_{\mathbb{M}}\int_{B(x,\varepsilon)}|f(x)-f(y)| \, \tilde{p}_t (x,y) \,d\mu(y)\,d\mu(x) \notag\\
=& \int_{\mathbb{M}}\int_{B(x,\varepsilon)}|f(x)-f(y)| \, \frac{1}{(4\pi t)^{n/2}}e^{-\frac{d(x,y)^2}{4t}} \,d\mu(y)\,d\mu(x) \notag\\
=& \int_{\mathbb{M}} \int_0^\varepsilon \int_{S^{n-1}} |f(x)-f(\exp_x(ru))| \, \frac{1}{(4\pi t)^{n/2}}e^{-\frac{r^2}{4t}} \, \theta_x(r,u) du \, dr \,d\mu(x),
\end{align}
Since $f \in C^2(\mathbb{M})$, there is a uniform Taylor expansion of order 1 around $x$,
\[
f(\exp_x(ru))=f(x) + r\left\langle \nabla f (x),u \right\rangle + r^2 H(x,ru),
\]
where $H$ is a bounded function. Moreover, applying the uniform exponential bound on the Jacobian $\theta_x(r,u)$ from~\eqref{E:Jacobian_exp_bound} we deduce on the one hand that
\begin{equation}\label{E:smooth case ana_02}
\lim_{t \to 0^+} \frac{1}{\sqrt{t}} \int_{\mathbb{M}} \int_0^\varepsilon \int_{S^{n-1}} r^2 |H(x,ru)| \, \frac{1}{(4\pi t)^{n/2}}e^{-\frac{r^2}{4t}} \, \theta_x(r,u) du \, dr \,d\mu(x)=0.
\end{equation}
On the other hand, the change of variables $\tilde{r}=r/\sqrt{t}$ yields
\begin{align*}
& \int_{\mathbb{M}} \int_0^\varepsilon \int_{S^{n-1}} |r\left\langle \nabla f (x),u \right\rangle| \, \frac{1}{(4\pi t)^{n/2}}e^{-\frac{r^2}{4t}} \, \theta_x(r,u) du \, dr \,d\mu(x) \\
&= t \int_{\mathbb{M}} \int_0^{\varepsilon/\sqrt{t}} \int_{S^{n-1}} |\tilde{r}\left\langle \nabla f (x),u \right\rangle| \, \frac{1}{(4\pi t)^{n/2}}e^{-\frac{\tilde{r}^2}{4}} \, \theta_x(\tilde{r}\sqrt{t},u) du \, d\tilde{r} \,d\mu(x) \\
&= \frac{1}{(4\pi )^{n/2}} \sqrt{t}   \int_{\mathbb{M}} \int_0^{\varepsilon/\sqrt{t}} \int_{S^{n-1}} \,  |\left\langle \nabla f (x),u \right\rangle|  \tilde{r}^n e^{-\frac{\tilde{r}^2}{4}} \, \frac{\theta_x(\tilde{r}\sqrt{t},u)}{(\tilde{r} \sqrt{t})^{n-1}} du \, d\tilde{r} \,d\mu(x).
\end{align*}
As in~\eqref{E:Jacobian_small_t} we have
\[
\frac{\theta_x(\tilde{r}\sqrt{t},u)}{(\tilde{r} \sqrt{t})^{n-1}}  \xrightarrow{t\to 0^+} 1
\]
uniformly on $x,\tilde{r}, u$ and therefore
\begin{multline*}
\lim_{t \to 0^+} \frac{1}{\sqrt{t}} \int_{\mathbb{M}} \int_0^\varepsilon \int_{S^{n-1}} |r\left\langle \nabla f (x),u \right\rangle| \, \frac{1}{(4\pi t)^{n/2}}e^{-\frac{r^2}{4t}} \, \theta_x(r,u) du \, dr \,d\mu(x) \\
= \frac{1}{(4\pi )^{n/2}}\int_{\mathbb{M}} \int_0^{+\infty} \int_{S^{n-1}} \,  |\left\langle \nabla f (x),u \right\rangle|  r^n e^{-\frac{r^2}{4}} \, du \, dr \,d\mu(x)  
=\frac{ 2 }{\sqrt{\pi}}\int_\mathbb{M} \| \nabla f \| d\mu.
\end{multline*}
Together with~\eqref{E:smooth case ana_01} and~\eqref{E:smooth case ana_02} the latter finally yields
\[
\lim_{t \to 0^+} \frac{1}{\sqrt{t}}\int_{\mathbb{M}}\int_{B(x,\varepsilon)}|f(x)-f(y)| \, \tilde{p}_t (x,y) \,d\mu(y)\,d\mu(x)=\frac{ 2 }{\sqrt
{\pi}}\int_\mathbb{M} \| \nabla f \| d\mu.
\]
\end{proof}

\subsection{General BV case}

The proof in the case of functions $f\in C^2(\mathbb{M})$ crucially relied in the possibility to perform Taylor approximation up to the correct order. Lowering the regularity of the functions this approach is not anymore available and it is now that Theorem~\ref{T:tildeHK_to_HK_ve} allows to overcome the problem. In this section we conclude the proof of the main Theorem~\ref{main} by obtaining suitable upper and lower estimates. We proceed first with the upper bound.

\begin{proposition}\label{P:limsup}
Let $f \in BV(\mathbb{M})$. For any $\delta >0$, there exists $\varepsilon >0$ such that
\[
\limsup_{t \to 0^+} \frac{1}{\sqrt{t}}\int_{\mathbb{M}}\int_{B(x,\varepsilon)}|f(x)-f(y)| \, \tilde{p}_t (x,y) \,d\mu(y)\,d\mu(x) \le (1+\delta) \frac{ 2 }{\sqrt
{\pi}} \| Df \| (\mathbb M).
\]
\end{proposition}

\begin{proof}
Let $\varepsilon_0 >0$ be smaller than the injectiviy radius of $\mathbb{M}$. Using polar coordinates exactly as in~\eqref{E:smooth case ana_01} we have
\begin{align*}
& \int_{\mathbb{M}}\int_{B(x,\varepsilon_0)}|f(x)-f(y)| \, \tilde{p}_t (x,y) \,d\mu(y)\,d\mu(x) \\
=& \int_{\mathbb{M}} \int_0^{\varepsilon_0} \int_{S^{n-1}} |f(x)-f(\exp_x(ru))| \, \frac{1}{(4\pi t)^{n/2}}e^{-\frac{r^2}{4t}} \, \theta_x(r,u) du \, dr \,d\mu(x), 
\end{align*}
where $\theta_x$ is the determinant of Jacobian of the exponential map in polar coordinates. Let us first assume that $f \in C^1(\mathbb{M})$. In this case, we can write 
\begin{align*}
|f(x)-f(\exp_x(ru))|& =r\left|  \int_0^1\langle \nabla f (\exp_x (sru)) , \para_{0,s}u \rangle ds \right| 
 \le r \int_0^1 \left| \langle \nabla f (\exp_x (sru)) , \para_{0,s}u \rangle  \right|ds.
\end{align*}
Let now $\delta >0$ and choose $0< \varepsilon_1 <\varepsilon_0$ small enough so that
\[
 \theta_x(r,u) \le (1+\delta) r^{n-1}
\]
holds uniformly for $ 0 \le r \le \varepsilon_1$ 
and thus
\begin{multline*}
\int_{\mathbb{M}}\int_{B(x,\varepsilon_1)}|f(x)-f(y)| \, \tilde{p}_t (x,y) \,d\mu(y)\,d\mu(x) \\
\le (1+\delta)  \int_{\mathbb{M}} \int_0^{\varepsilon_1} \int_{S^{n-1}} \int_0^1 \left| \langle \nabla f (\exp_x (sru)) , \para_{0,s}u \rangle  \right| \, \frac{1}{(4\pi t)^{n/2}}e^{-\frac{r^2}{4t}} \, r^n \, ds \, du \, dr \,d\mu(x).
\end{multline*}
In the latter integral, we now perform the change of variables 
\[
(x,u) \mapsto ( \exp_x (sru) , \para_{0,s}u)
\]
and choose $0< \varepsilon_2 <\varepsilon_1$ so that the inverse of the Jacobian of this map is bounded above by $1+\delta$ uniformly for $0 \le r \le \varepsilon_2$.
 Putting $z=\exp_x (sru)$ and $v=\para_{0,s}u$, 
this yields
\begin{align*}
& \int_{\mathbb{M}}\int_{B(x,\varepsilon_2)}|f(x)-f(y)| \, \tilde{p}_t (x,y) \,d\mu(y)\,d\mu(x) \\
&\le  (1+\delta)^2  \int_{\mathbb{M}} \int_0^{\varepsilon_2} \int_{S^{n-1}} \int_0^1 \left| \langle \nabla f (z) , v \rangle  \right| \, \frac{1}{(4\pi t)^{n/2}}e^{-\frac{r^2}{4t}} \, r^n \, ds \, dv \, dr \,d\mu(z) \\
&\le  (1+\delta)^2 \sqrt{t} \frac{ 2 }{\sqrt{\pi}}\int_\mathbb{M} \| \nabla f \| d\mu.
\end{align*}
The latter inequality holds for $f \in C^1(\mathbb{M})$ and by approximation also for $f \in BV(\mathbb M)$. Indeed, see e.g.~\cite[Proposition 1.4]{MPPP07b}, for any $f\in BV(\mathbb{M})$ there is $\{f_n\}_{n\geq 0} \subset C^\infty(\mathbb{M})$ with $f_n \to f$ in $L^1$ and $\int_\mathbb{M} \| \nabla f_n \| d\mu \to \| Df \| (\mathbb M)$. Therefore, we conclude that
\[
\limsup_{t \to 0^+} \frac{1}{\sqrt{t}}\int_{\mathbb{M}}\int_{B(x,\varepsilon_2)}|f(x)-f(y)| \, \tilde{p}_t (x,y) \,d\mu(y)\,d\mu(x) \le (1+\delta)^2 \frac{ 2 }{\sqrt{\pi}} \| Df \| (\mathbb M)
\]
holds for any $f\in BV(\mathbb{M})$.
\end{proof}

The final ingredient we need is a suitable lower bound. In this case we obtain it by smoothing out $f$ and then use Proposition~\ref{P:smooth case ana}. 

\begin{proposition}\label{P:liminf}
Let $f \in BV(\mathbb{M})$. For any $\delta >0$, there exists $\varepsilon >0$ such that
\[
\liminf_{t \to 0^+} \frac{1}{\sqrt{t}}\int_{\mathbb{M}}\int_{B(x,\varepsilon)}|f(x)-f(y)| \, \tilde{p}_t (x,y) \,d\mu(y)\,d\mu(x) \ge (1-\delta) \frac{ 2 }{\sqrt{\pi}}\| Df \| (\mathbb M).
\]
\end{proposition}

\begin{proof}
Let $f \in BV(\mathbb M)$, fix $\delta>0$ and let $\varepsilon_0 >0$ be smaller than the injectivity radius of $\mathbb{M}$. For $s>0$, consider the function $f_{s,\varepsilon_0}$ defined on $\mathbb M$ by
\begin{equation*}
f_{s,\varepsilon_0}(x):= C(s,\varepsilon_0) \int_{B_x(0,\varepsilon_0) } e^{-\frac{\| \xi \|^2}{s}} f (\exp_x (\xi)) d\xi,
\end{equation*}
where $C(s,\varepsilon_0)$ is such that $C(s,\varepsilon_0) \int_{B_x(0,\varepsilon_0) } e^{-\frac{\| \xi \|^2}{s}}  d\xi=1$. Since $f_{s,\varepsilon_0}$ is essentially a convolution with a smooth kernel, $f_{s,\varepsilon_0} \in C^\infty (\mathbb{M})$. Also, note that if $y \in \mathbb{M}$ with $d(x,y) \le \varepsilon_0$ we also have 
\[
f_{s,\varepsilon_0}(y)= C(s,\varepsilon_0) \int_{B_x(0,\varepsilon_0) } e^{-\frac{\| \xi \|^2}{s}} f (\exp_y ( \para_{x,y} \, \xi)) d\xi,
\]
where $B_x(0,\varepsilon_0)$ denotes a ball in $T_x\mathbb{M}$ and $\para_{x,y} \, \xi$ denotes the parallel transport of $\xi $ from $T_x\mathbb{M}$ to $T_{y} \mathbb{M}$ along the unique length parametrized geodesic joining $x$ to $y$. We may now chose $0<\varepsilon_1 <\varepsilon_0$ so that 
\[
(1-\delta) d(x,y) \le d\big(\exp_x (\xi),  \exp_y ( \para_{x,y} \, \xi)\big) \le (1+\delta) d(x,y)
\]
holds uniformly for all $d(x,y) \le \varepsilon_1$ and $\| \xi \| \le \varepsilon_1$. 
Thus,
\begin{align*}
 & \int_{\mathbb{M}}\int_{B(x,\varepsilon_1)}|f_{s,\varepsilon_1}(x)-f_{s,\varepsilon_1}(y)| \, \tilde{p}_t (x,y) \,d\mu(y)\,d\mu(x) \\
 \le & C(s,\varepsilon_1)\int_{\mathbb{M}}\int_{B(x,\varepsilon_1)} \int_{B_x(0,\varepsilon_1) } e^{-\frac{\| \xi \|^2}{s}} | f (\exp_x (\xi))-f (\exp_y ( \para_{x,y} \, \xi) )|  \, \tilde{p}_t (x,y) \, d\xi \, d\mu(y)\,d\mu(x).
\end{align*}
In the latter integral, we perform the change of variables
\[
(x,y) \to ( \exp_x (\xi) , \exp_y ( \para_{x,y} \, \xi))
\]
and choose $0< \varepsilon_2 <\varepsilon_1$ so that, uniformly for $0 \le r \le \varepsilon_2$, the inverse of the Jacobian of this map is bounded above by $1+\delta$. Then, setting $z_1=\exp_x (\xi)$ and $z_2=\exp_y ( \para_{x,y} \, \xi)$ we have
\begin{align*}
 & \int_{\mathbb{M}}\int_{B(x,\varepsilon_2)}|f_{s,\varepsilon_2}(x)-f_{s,\varepsilon_2}(y)| \, \tilde{p}_t (x,y) \,d\mu(y)\,d\mu(x) \\
 \le & (1+\delta)C(s,\varepsilon_2)\int_{\mathbb{M}}\int_{B(z_1,(1+\delta)\varepsilon_2)} \int_{B_x(0,\varepsilon_2) } e^{-\frac{\| \xi \|^2}{s}} | f (z_1)-f (z_2) |  \, \frac{e^{-\frac{(1-\delta)^2d(z_1,z_2)^2}{4t}}}{(4\pi t)^{n/2}} \, d\xi \, d\mu(z_2)\,d\mu(z_1)  \\
 \le & (1+\delta) \int_{\mathbb{M}}\int_{B(z_1,(1+\delta)\varepsilon_2)}   | f (z_1)-f (z_2) |  \, \frac{1}{(4\pi t)^{n/2}}e^{-\frac{(1-\delta)^2d(z_1,z_2)^2}{4t}}  \, d\mu(z_2)\,d\mu(z_1) \\
  \le & \frac{1+\delta}{(1-\delta)^n} \int_{\mathbb{M}}\int_{B(z_1,(1+\delta)\varepsilon_2)}   | f (z_1)-f (z_2) |  \, \tilde{p}_{t/(1-\delta)^2}(z_1,z_2) \, d\mu(z_2)\,d\mu(z_1)
\end{align*}
which implies that
\begin{multline*}
  \liminf_{t \to 0^+} \frac{1}{\sqrt{t}} \int_{\mathbb{M}}\int_{B(z_1,(1+\delta)\varepsilon_2)}   | f (z_1)-f (z_2) |  \, \tilde{p}_{t}(z_1,z_2) \, d\mu(z_2)\,d\mu(z_1) \\
 \ge  \frac{(1-\delta)^{n+1}}{(1+\delta)}  \liminf_{t \to 0^+} \frac{1}{\sqrt{t}} \int_{\mathbb{M}}\int_{B(x,\varepsilon_2)}|f_{s,\varepsilon_2}(x)-f_{s,\varepsilon_2}(y)| \, \tilde{p}_t (x,y) \,d\mu(y)\,d\mu(x).
\end{multline*}
In view of Proposition~\ref{P:smooth case ana} we also know that
\[
\liminf_{t \to 0^+} \frac{1}{\sqrt{t}} \int_{\mathbb{M}}\int_{B(x,\varepsilon_2)}|f_{s,\varepsilon_2}(x)-f_{s,\varepsilon_2}(y)| \, \tilde{p}_t (x,y) \,d\mu(y)\,d\mu(x)=\frac{ 2 }{\sqrt
{\pi}}\int_\mathbb{M} \| \nabla f_{s,\varepsilon_2} \| d\mu
\]
and hence
\begin{multline*}
\liminf_{t \to 0^+} \frac{1}{\sqrt{t}} \int_{\mathbb{M}}\int_{B(z_1,(1+\delta)\varepsilon_2)}   | f (z_1)-f (z_2) |  \, \tilde{p}_{t}(z_1,z_2) \, d\mu(z_2)\,d\mu(z_1) \\
 \ge \frac{(1-\delta)^{n+1}}{(1+\delta)} \frac{ 2 }{\sqrt{\pi}}\int_\mathbb{M} \| \nabla f_{s,\varepsilon_2} \| d\mu.
\end{multline*}
Noting that $f_{s,\varepsilon_2} \xrightarrow{s \to 0} f $ in $L^1(\mathbb{M},\mu)$ and $f\in BV(\mathbb{M})$, by virtue of~\eqref{E:BV_Lip_char} we conclude 
\begin{multline*}
  \liminf_{t \to 0^+} \frac{1}{\sqrt{t}} \int_{\mathbb{M}}\int_{B(z_1,(1+\delta)\varepsilon_2)}   | f (z_1)-f (z_2) |  \, \tilde{p}_{t}(z_1,z_2) \, d\mu(z_2)\,d\mu(z_1) \\
 \ge  \frac{(1-\delta)^{n+1}}{(1+\delta)} \frac{ 2 }{\sqrt{\pi}} \| Df \| (\mathbb M).
\end{multline*}
\end{proof}

We are now ready to prove Theorem~\ref{main} and in particular to compute the variation of a function $f \in BV( {\mathbb M})$ as
\[
\lim_{t \to 0^+} \frac{1}{\sqrt{t}}  \int_{\mathbb M} P_t (|f-f(x)|)(x) d\mu(x) =  \frac{ 2 }{\sqrt{\pi}} \| Df \| (\mathbb M).
\]

\begin{proof}[Proof of Theorem~\ref{main}]
Combining Propositions~\ref{P:limsup} and~\ref{P:liminf} and Theorem~\ref{T:tildeHK_to_HK_ve}, for any $\delta>0$ we have
\[
\limsup_{t \to 0^+} \frac{1}{\sqrt{t}}\int_{\mathbb{M}}\int_{\mathbb M}|f(x)-f(y)| \, p_t (x,y) \,d\mu(y)\,d\mu(x) \le (1+\delta) \frac{ 2 }{\sqrt{\pi}} \| Df \| (\mathbb M)
\]
and
\[
\liminf_{t \to 0^+} \frac{1}{\sqrt{t}}\int_{\mathbb{M}}\int_{\mathbb M}|f(x)-f(y)| \, p_t (x,y) \,d\mu(y)\,d\mu(x) \ge (1-\delta) \frac{ 2 }{\sqrt{\pi}} \| Df \| (\mathbb M).
\]
Letting $\delta \to 0$ the assertion follows.
\end{proof}

\section{Sobolev spaces: $p>1$}\label{S:main_Sobolev}
As it may have become apparent to the reader, all arguments used in the analytic approach from Section~\ref{S:Analytic_approach} apply also when $p>1$, thus yielding a corresponding characterization of Sobolev spaces $W^{1,p}(\mathbb{M})$. This characterization goes along the lines of~\cite[Theorem 3.3]{ARB20} and here we can provide the convergence of the $p$-variation. The proof can be carried out almost verbatim to the previous section and thus details are left to the interested reader.

\begin{theorem}
Let $\mathbb M$ be a compact Riemannian manifold of dimension $n$. For any $p> 1$,
\[
W^{1,p}(\mathbb{M})=\Big\{f\in L^p(\mathbb{M},\mu)\colon\limsup_{t\to 0^+}\frac{1}{\sqrt{t}}\int_{\mathbb{M}}P_t(|f-f(x)|^p)(x)\,d\mu(x)<\infty\Big\}.
\]
Moreover, when the limit exists it satisfies
\begin{equation}\label{E:Sobolev_sn_char}
\lim_{t \to 0^+} \frac{1}{t^{p/2}} \int_{\mathbb M} P_t (|f-f(x)|^p)(x) d\mu(x) =  \frac{2^p}{\sqrt{\pi}}\Gamma\Big(\frac{1+p}{2}\Big)\int_{\mathbb{M}} \|\nabla f \|^p d\mu.
\end{equation}
\end{theorem}

Since $\frac{2^p}{\sqrt{\pi}}\Gamma\Big(\frac{1+p}{2}\Big)=\frac{1}{\mathcal{H}^{n-1}(S^{n-1})}\int_{S^{n-1}}|\langle e, v\rangle|^p\,dv$ for a unit vector $e\in S^{n-1}$, we have that~\eqref{E:Sobolev_sn_char} in particular recovers~\cite[Theorem 1.1]{KM19}. Also, in this case it is known that $f\in W^{1,p}(\mathbb{M})$ if and only if there exists a sequence $\{f_n\}_{n\geq 1}\subset C^\infty(\mathbb{M})$ for which $f_n\to f$ in $L^p(\mathbb{M},\mu)$ and $\int_{\mathbb{M}}\|\nabla f_n\|^pd\mu\to\int_{\mathbb{M}}|\nabla f\|^pd\mu$, see e.g.~\cite[Proposition 3.2]{Heb99}.

\bibliographystyle{amsplain}
\bibliography{Riemannian_BV}
\end{document}